\documentclass[12pt]{amsart}
\usepackage{amssymb,latexsym}
\usepackage{enumerate}

\makeatletter
\newtheorem{thm}{Theorem}[section]
\newtheorem{prop}[thm]{Proposition}

\theoremstyle{definition}
\newtheorem{defin}[thm]{Definition}
\newtheorem{rem}[thm]{Remark}

\newcommand{\C}{\mathbb{C}}
\newcommand{\T}{\mathcal{T}}

\begin{document}

\baselineskip=17pt

\title{A remark on projective limits of function spaces}
\author{A. Kampoukou and V. Nestoridis}
\address{ Department of Mathematics\\National and Kapodistrian University of Athens \\ Panepistemiopolis\\15784 Athens\\ Greece}
\email{kampouk@math.uoa.gr}
\email{ vnestor@math.uoa.gr }

\begin{abstract}
Let $\Omega \subset \C^d$ be an open set and $K_m , m=1,2,\dots$ an exhaustion of $\Omega$ by compact subsets of $\Omega$.
 We set $\Omega_m=K_m^o$ and let $X_m(\Omega_m)$ be a topological space of holomorphic functions on $\Omega_m$ between $A^\infty (\Omega_m)$ and $H(\Omega_m)$.
 Then we show that the projective limit of the family $X_m(\Omega_m)$, $m=1,2,\dots$,  under the restriction maps is homeomorphic and linearly isomorphic to the Fr\'{e}chet space $H(\Omega)$, idependently of the choice of the spaces $X_m(\Omega_m)$.
\end{abstract}



\maketitle
\section{Introduction}  
Let $Y$ be a space of holomorhic functions on a domain $\Omega$ in $\C^d, d=1,2,\dots$ . For most of the spaces $Y$ the following holds. If for every bump $V$ of $\Omega$ there exists a function $f_V$ in $Y$ which is not holomorphically extendable to $V$, then there exists a function $f$ in $Y$ not extendable to any $V$. Moreover, the set of such functions $f$ in $Y$ is dense and $G_{\delta}$ in $Y$.  (see \cite{JP}, \cite{N})

Examples of such spaces are the space $H(\Omega)$ of all holomorphic functions in $\Omega$ endowed with the topology of uniform convergence on compacta, the spaces $A^p(\Omega)$ and $H_p^\infty (\Omega)$, Bergman spaces, Hardy spaces etc, provided that $\Omega$ satisfies some good assumptions. 

Most of these spaces endowed with their natural topology, contain $A^{\infty}(\Omega)$ and they are contained in $H(\Omega)$ :  $A^{\infty}(\Omega) \subset Y \subset H(\Omega)$. Furthermore, the two inclusion maps $i_1\colon A^{\infty}(\Omega) \to Y$, $i_2 \colon Y \to H(\Omega)$ with $i_1(f)=f$ and $i_2(w)=w$ are continuous. We denote by $M(\Omega)$ the set of such topological function spaces on $\Omega$. 

Let $K_m$, $m=1,2, \dots$, be an exhaustion by compact subsets of the open set $\Omega \subset {\C}^d$ and $\Omega_m=K_m^o$. Let $X_m(\Omega_m)$ be any topological function space belonging to $M(\Omega_m)$. Let $X$ denote the projective (inverse) limit of the family $X_m(\Omega_m) , m=1,2,\dots$ where $T_{m,n}:X_m(\Omega_m) \to X_n(\Omega_n) , n \leq m$ is the restriction map $ f \to f|_{\Omega_n}$.

Then we show that $X$ is a topological vector space homeomorphic and linearly isomorphic to the Fr\'{e}chet space $H(\Omega)$ of all holomorphic functions on $\Omega$ endowed with the topology of uniform convergence on all compact subsets of $\Omega$, independently of the choice of the topological spaces  $X_m(\Omega_m)$ from  $M(\Omega_m)$. This holds even if the spaces  $X_m(\Omega_m)$ are neither complete, nor vector spaces.

\section{About projective limits}
We need the following particular case of projective limits (see  \cite{E}).

\begin{defin}
Let $(Y_n , \T_n ) , n=1,2,\dots$ be a sequence of topological spaces. For $n<m$ let $T_{m,n} \colon Y_m \to Y_n$ be a continuous map and let $T_{n,n}$ be the identity map from $Y_n$ to itself. We assume that for $m \leq n \leq s$ it holds $T_{s,m}=T_{n,m} \circ T_{s,n}$. 

Then the projective (or inverse) limit of this family is defined to be the following subset $Y$ of the cartesian product $\prod _{n=1}^{\infty} Y_n$ where
\[Y=\left \lbrace g=(g_1, g_2, \dots) : \text{ for every } n \leq m \text { it holds } T_{m,n}(g_m)=g_n \right\rbrace \]
 The topology $\T_Y$ on $Y$ we consider is the relative topology from  $\prod _{n=1}^{\infty} Y_n$ endowed with the product topology.
\end{defin}

\begin{prop}  \cite{E}
Under the above notation and assumption a net $g^i=(g_1^i , g_2^i  \dots) \in Y,  \in I$ converges to $g=(g_1,g_2,\dots) \in Y $ in the topology $\T_Y$ of Y if and only if $g^i_k \to g_k$ in the topology $T_k$ of $Y_k$ for all $k=1,2, \dots$.
\end{prop}

\section{Some function spaces}
Let $\Omega$ be an open subset of $\C^d$, $d \geq 1$. We denote by $H(\Omega)$ the space of holomorphic functions on $\Omega$ endowed with the topology of the uniform convergence on compacta.
 This is a Fr\'{e}chet space with seminorms $\Vert f \Vert _m = \sup_{z \in K_m} \vert f(z) \vert , m=1,2, \dots$ where $K_m$ is an exhaustion of $\Omega$ by compact subsets of $\Omega$; for instance we can take (see \cite{R}, \cite{G})
 \[K_m= \left \lbrace z \in \Omega : \Vert z \Vert \leq m \text{ and } dist(z, \Omega^c) \geq  \frac{1}{m} \right \rbrace\]
 In what follows we consider that $G$ is a bounded open subset of  $\C^d$ , although the spaces we will consider, can also be defined for an unbounded open set $G$, with some modifications in the definition of their natural topology.
 
 The space $A^{\infty}(G)$ is defined as the set of holomorphic functions on $G$ such that each mixed partial derivative of $f$ extends continuously on the closure $\overline{G}$ of $G$. The natural topology of $A^{\infty}(G)$ is defined by the seminorms 
 \[sup_{z \in G} \vert \frac{\partial^{a_1+ \dots +a_d} f}{\partial {z_1}^{a_1} \dots \partial {z_d}^{a_d}}(z)  \vert , a_1 ,\dots, a_d \in \left \lbrace 0,1,2, \dots \right \rbrace \text{, where } z=(z_1, \dots , z_d).\]
  This is also a Fr\'{e}chet space .
 
 Next we are interested in spaces $X(G)$ containing $A^{\infty}(G)$ and contained in $H(G)$, i.e. $A^{\infty}(G) \subset Y \subset H(G)$.
 We also require that the injections are continuous; that is, if a net  $f^i \in A^{\infty}(G)$, $i \in I$,  converges in the topology of  $A^{\infty}(G)$  to a function $f \in  A^{\infty}(G)$, then  $f^i \to f$ in the topology of $X(G)$, as well; and if a net $w^i \in X(G)$, $i \in I$, converges to $w \in X(G)  $ in the topology of $X(G)$ then $w^i \to w$ in the topology of $H(G)$.
We denote by $M(G)$ the set of all topological spaces as above.
 
Such spaces are the spaces $A^p(G)$ consisted of $f \in H(G)$ of which every mixed partial derivative of order less than or equal to $p$ extends continuously on $ \overline{G}$ and 
the spaces $H^p_\infty (G)$  consisted of $f \in H(G)$ of which every mixed partial derivative of order less than or equal to $p$ is bounded on  $ \overline{G}$.
The above two spaces are Banach spaces, provided $p<+\infty$ and Fr\'{e}chet spaces if $p=+\infty$. 
Their topologies are defined by the seminorms
\[sup_{z \in G} \vert \frac{\partial^{a_1+ \dots +a_d} f}{\partial {z_1}^{a_1} \dots \partial {z_d}^{a_d}}(z) \vert  ,  a_1+ \dots +a_d \leq p \text{, where }  z=(z_1, \dots , z_d).\]

Another example of such a space is the p-Bergman space, $0<p<+\infty$, containing all holomorphic functions $f$ on $G$, such that $\int _G \vert f \vert ^p d\nu < \infty$, where $d\nu$ denotes the Lebesgue measure on $G \subset \C^d$.  (see\cite{DS})

Other examples are Hardy spaces on the disc or the ball or on more general domains (see \cite{D} , \cite{R2} \cite{S}) , as well as subclasses of the Nevanlinna class , BMOA, VMOA, Dirichlet spaces on the disc etc.

All the previous spaces are complete spaces . However, we can also consider non complete spaces containing  $A^{\infty}(G)$ and contained in $H(G)$ such  that the two injections are continuous. 
For instance, consider the set $Y=A^{\infty}(G)$ with the relative topology induced by $H(G)$; that is, for $f_n , f \in A^{\infty}(G)$ the convergence in the new space is the uniform convergence on compacta $f_n \to f$. 
This space contains $A^{\infty}(G)$, is contained in  $H(G)$ and the injections $i_1:A^{\infty}(G) \to Y$, $ i_2 \colon Y \to H(G)$ with $i_1(f)=f$ and $i_2(g)=g$ are continuous where $A^{\infty}(G)$ and  $H(G)$ are endowed with their natural topologies and $Y$ with the relative topology from $H(G)$. 
But this space is not complete, in general. 
For instance, if $(\overline {G})^o =G$ or if $G$ is a bounded, open subset of $\C$, then $A^{\infty}(G)$ with the relative topology induced by $H(G)$  is not complete.

It is also possible that a space belongs to $M(G)$ although it is not even a vector space. To give such an example we consider the set
\[Y=A^{\infty}(G) \cup \left\lbrace f \colon G \to \C \text{ holomorphic such that } \vert f(z) \vert <1 \text{, for all } z\in G \right\rbrace , \]
endowed with the topology of uniform convergence on $G$. Then $A^{\infty}(G) \subset Y \subset H(G)$ and if $f_n , f \in A^{\infty}(G)$ are such that $f_n \to f$ in the topology of $ A^{\infty}(G)$, it follows trivially that  $f_n \to f$ in the topology of $Y$. Also if $\phi_n , \phi \in Y$ are such that $\phi_n \to \phi$ in the topology of $Y$, then obviously the convergence is uniform on each compact subset of $G$, as well.

\section{The result}
We start with the following definition

\begin{defin}
Let $G$ be a bounded open subset of $\Omega^d$. We denote by $M(G)$ the set of all topological spaces $X(G)$ satisfying $A^{\infty}(G) \subset X(G) \subset H(G)$ and such that the two injections $i_1 \colon  A^{\infty}(G) \to X(G)$, $i_1(f)=f$ and $i_2 \colon X(G) \to H(G)$, $i_2(g)=g$ are continuous; that is, if a net $f^i \in A^{\infty} (G)$, $i \in I$, converges to $f \in A^{\infty}(G)$, in the topology of $A^{\infty}(G)$, then $f^i \to f$ in the topology of $X(G)$, as well. Moreover, if a net $w^i \in X(G)$ converges to $w \in X(G)$,in the topology of $X(G)$, then $w_i \to w$ uniformly on compacta of $G$.
\end{defin}

 Let $\Omega$ be an open subset of $C^d$, bounded or unbounded. 
 Let $K_m , m=1,2,\dots$ be an exhausting family of compact subsets of $\Omega$; that is, $K_m \subset K_{m+1}^o$ and $\cup_{m=1}^{\infty} K_m =\Omega$ (see \cite{R}),(\cite{G}).
  We set $\Omega_m =K_m^o$; then $ \Omega_m \subset \overline{\Omega_m} \subset \Omega_{m+1}$. For each $m=1,2, \dots$ let $X_m(\Omega_m)$ be a space in $M(\Omega_m)$.
 
 For $n \leq m$ we consider the restriction map $T_{m,n} \colon X_m(\Omega_m) \to X_n(\Omega_n)$ given by $T_{m,n}(f)=f| _{\Omega_n}$, for every $f \in X_m(\Omega_m)$.
  These maps are well defined and continuous. Indeed, $X_m(\Omega_m) \subset H(\Omega_m)$ and $\overline{\Omega_n} \subset \Omega_m$; thus, for every $f \in X_m(\Omega_m)$, it follows that $T_{m,n}(f)=f| _{\Omega_n} \in A^{\infty}(\Omega_n) \subset X_n(\Omega_n)$, provided that $n<m$. 
 Also if a net $f^i \in X_m(\Omega_m)$ converges to $f \in X_m(\Omega_m)$, in the topology of $X_m(\Omega_m)$, then by assumption, it converges uniformly on compacta of $\Omega_m$. 
 Weierstrass theorem yields uniform convergence on compacta of every mixed partial derivative. 
 Since $\overline{\Omega_n}$ is a compact subset of $\Omega_m$, it follows that $f^i \mid_{\Omega_n} \to f|_{\Omega_n}$ in the topology of $A^{\infty}(\Omega_n)$. 
 By assumption, this implies $f^i \to f$ in the topology of $ X_n(\Omega_n)$. Therefore, $T_{m,n} :X_m(\Omega_m) \to X_n(\Omega_n)$ is continuous. One can also very easily verify that $T_{m,n}=T_{s,n} \circ T_{m,s}$, for all $n \leq s \leq m$.
 Therefore, we can define the projective limit $X$, which is a topological space (see \cite{E}).
 
 \begin{thm}
 Under the above assumptions and notation the topological space $X$ is homeomorphic to $H(\Omega)$.
 \end{thm}
 
 \begin{proof}
 Let $f \in H(\Omega)$; then $f|_{\Omega_m} \in A^{\infty}(\Omega_m) \subset X_m(\Omega_m)$.
  Thus, if we set $g_m=f|_{\Omega_m}$ and $g=(g_1 , g_2, \dots)=g(f)$, then $g \in \prod_{m=1}^{\infty}X_m(\Omega_m)$.
 For $n \leq m$ one can easily verify that $T_{m,n}(g_m)=g_n$. Therefore, $g \in X$.  
  In this way we have defined a map $\Phi \colon H(\Omega) \to X$, where $\Phi(f)=g(f)$.
  
  This map is one to one. Indeed, if $g(f)=g(w)$ it follows that $g |_{\Omega_m} = w|_{\Omega_m}$ for all $m=1,2,\dots$. Since $\cup_m \Omega_m=\Omega $, it follows that $f(z)=w(z)$ for all $z \in \Omega$; thus $f=w$. This proves that the map $\Phi$ is one to one.
  
  The map is also onto. Indeed, if $g=(g_1,g_2,\dots) \in X$, then for every $z \in \Omega$ there exists a least natural number $n(z)$ so that $z \in \Omega_m$, for all $m \geq n(z)$, because $\Omega_m \subset \Omega_{m+1}$ and $\cup_m \Omega_m =\Omega$. 
 Futher, since $T_{m,n(z)}g_m=g_{n(z)}$, it follows that $g_m(z)=g_{n(z)}(z)$, for all $m \geq n(z)$.
  We set $f(z)=g_{n(z)}(z)$. We will show that this function is holomorphic on $\Omega$. Indeed, let $z_0 \in \Omega$; then $z_0 \in \Omega_{n(z_0)}$. 
  Since $\Omega_{n(z_0)}$ is open, there exists a ball $B(z_0 , r)$ contained in $\Omega_{n(z_0)}$. For each $z \in B(z_0 , r)$, we have $n(z) \leq n(z_0)$ and $f(z)=g_{n(z)}(z)=g_{n(z_0)}(z)$. But $g_{n(z_0)} \in X_{n(z_0)}(\Omega_{n(z_0)}) \subset H(\Omega_{n(z_0)})$. It follows that f is holomorphic in $B(z_0 , r)$. Since $z_0$ is arbitary in $\Omega$, it follows that $f \in H(\Omega)$. It is immediate that $\Phi(f) = g$. Since $g$ is arbitrary in $X$, it follows that the map $\Phi$ is onto.
  
 Further, $\Phi$ is continuous.
 Indeed, if a net $f^i \in H(\Omega)$, $i \in I$ converges to $f \in H(\Omega)$,  uniformly on compacta, by Weierstrass theorem the same holds for every mixed partial derivative. Since $\overline{\Omega_m}$ is a compact subset of $\Omega$ it follows that $f^i\mid_{\Omega_m} \to f|_{\Omega_m}$ in the topology of $A^{\infty}(\Omega_m)$. By assumption this implies the convergence $f^i|_{\Omega_m} \to f \mid_{\Omega_m}$ in the topology of $X_m(\Omega_m)$. According to Proposition 2.2, it follows $\Phi(f^i) \to \Phi(f)$ in the topology of $X$. Thus, the map $\Phi$ is continuous.
  
  Finally, we will show that $\Phi^{-1}$ is also continuous. Suppose that $\Phi(f^i)=(f^i|_{\Omega_1} ,  f^i|_{\Omega_2}, \dots) \to \Phi(f) = (f |_{\Omega_1} ,  f|_{\Omega_2}, \dots)$ in the topology of $X$. 
 According to Proposition 2.2  this implies $f^i|_{\Omega_m} \to f|_{\Omega_m}$ for all $m=1,2, \dots$ in the topology of $X_m(\Omega_m)$, which, by assumption, implies the uniform convergence on compacta of $\Omega_m$.
 We will show that $f^i \to f$ uniformly on each compact set $K \subset \Omega$.
 Let $K$ be such a compact set. 
 Since $K \subset \Omega= \cup_{m=1}^{\infty}\Omega_m$ and the open sets $\Omega_m$ satisfy $\Omega_m \subset \Omega_{m+1}$, by compactness of $K$ it follows that there exists $m_0$ such that $K \subset \Omega_{m_0}$. 
 But $f^i|_{\Omega_{m_0}} \to f|_{\Omega_{m_0}}$ uniformly on each compact of $\Omega_{m_0}$; in particular on $K$. Thus, $ f^i \to f$ uniformly on each compact set $K \subset \Omega$ and therefore, $f^i \to f$ in the topology of $H(\Omega)$. Thus, $\Phi^{-1}$ is continuous, as well and $\Phi$ is a homeomorphism. This completes the proof.
 \end{proof}
 
\begin{rem}
It is well known and easy to verify that the denumerable projective limit of Fr\'{e}chet spaces is also a Fr\'{e}chet space. In our case $X$ is homeomorphic to the Fr\'{e}chet space $H(\Omega)$, although the spaces $X_m(\Omega_m)$ may not be Fr\'{e}chet spaces. Furthermore, independently of the fact that the spaces $X_m(\Omega_m)$ are topological vector spaces or they do not have linear structure, one can easily verify that $X$ is always a topological vector space and that the map $\Phi$ is a linear isomorphism. 
\end{rem} 
 
 \begin{rem}
 Let us consider the compact set 
 \[K_t= \left\lbrace z \in \Omega \colon \vert z \vert \leq t , dist (z, \Omega^c) \geq \frac{1}{t} \right\rbrace \text{, } t \geq t_0,\]
 for a  convenient  $t_0 >0$ such that $K_{t_0}^o \neq \emptyset$. Then for  $t<s$ it follows $K_t \subseteq K_s^o$ and $\cup _{t \geq t_0}K_t=\Omega$.
 We set $\Omega_t =K_t^o$ and let $X_t(\Omega_t)$ be any element of $M(\Omega_t)$. 
 One can prove that the projective limit of the family $X_t(\Omega_t)$ as $t \to +\infty$, where $T_{s,t}$ is the restriction map $f \to f|_{\Omega_t}$, $t<s$, is homeomorphic and linearly isomorphic to $H(\Omega)$. The proof is similar to that of Theorem 4.2.
 \end{rem}
 
\subsection*{Acknowledgement} 
We would like to thank M.Fragoulopoulou for her interest in this work.

 \end{document}